\documentclass[12pt,bezier]{article}
\usepackage{times}
\usepackage{booktabs}
\usepackage{pifont}
\usepackage{floatrow}
\usepackage{caption}
\usepackage{mathrsfs}
\usepackage[fleqn]{amsmath}
\usepackage{amsfonts,amsthm,amssymb,mathrsfs,bbding}
\usepackage{txfonts}
\usepackage{graphics,multicol}
\usepackage{graphicx}
\usepackage{color}
\usepackage{caption}
\usepackage{cite}
\usepackage{latexsym,bm}
\usepackage{indentfirst}
\usepackage{xcolor}
\usepackage[colorlinks=true,
linkcolor=blue,citecolor=blue,
urlcolor=blue]{hyperref}

\evensidemargin=\oddsidemargin\topmargin=-1.5cm

\newtheorem{prob}{Problem}[section]
\newtheorem{lem}{Lemma}[section]
\newtheorem{thm}{Theorem}[section]

\theoremstyle{definition}

\addtocounter{section}{0}

\begin{document}
\title{A Max-Min problem on spectral radius and connectedness of graphs\footnote{Supported
by National Natural Science Foundation of China (Nos.\,12061074 and 12271362)}}
\author{
{\small Zhenzhen Lou$^{a}$,\ \ Changxiang He$^{a}$\footnote{Corresponding author. \\Email addresses: louzz@usst.edu.cn (Z. Lou), changxianghe@hotmail.com (C. He).}
}\\[2mm]
\footnotesize $^a$  College of Science, University of Shanghai for Science and Technology,
 Shanghai, 200093, China
}

\date{}
\maketitle {\flushleft\large\bf Abstract}
In the past decades, many scholars concerned
which edge-extremal problems have spectral analogues?
Recently, Wang, Kang and Xue
showed an interesting result on $F$-free graphs
[J. Combin. Theory Ser. B 159 (2023) 20--41].
In this paper, we study the above problem on critical graphs.
Let $P$ be a property defined on a family $\mathbb{G}$ of graphs.
A graph $G$ in $\mathbb{G}$ is said to be $P$-critical,
if it has the property $P$ but $G-e$ no longer has for any edge $e\in E(G)$.
Especially, a graph is minimally $k$-(edge)-connected,
if it is $k$-connected (respectively, $k$-edge connected)
and deleting an arbitrary edge always leaves a graph which is not
$k$-connected (respectively, $k$-edge-connected).
An interesting Max-Min problem asks what is the maximal spectral radius
of an $n$-vertex minimally $k$-(edge)-connected graphs?
In 2019, Chen and Guo [Discrete Math. 342 (2019) 2092--2099] gave the answer for $k=2$.
In 2021, Fan, Goryainov and Lin [Discrete Appl. Math.
305 (2021) 154--163] determined the extremal spectral radius
for minimally $3$-connected graphs.
We obtain some structural properties
of minimally $k$-(edge)-connected graphs.
Furthermore, we solve the above Max-Min problem for $k\geq3$,
which implies that every minimally $k$-(edge)-connected graph with maximal
spectral radius also has maximal number of edges.
Finally, a general problem is posed for further research.

\vspace{0.1cm}
\begin{flushleft}
\textbf{Keywords:} minimally $k$-(edge)-connected graph;
maximal spectral radius; connectivity
\end{flushleft}
\textbf{AMS Classification:} 05C50\ \ 05C75

\section{Introduction}

Perhaps the most basic property a graph may process is that of being connected. At a more refined level, there are various functions that may be said to measure the connectedness of a connected graph \cite{BB-1}.
A graph is said to be \emph{connected} if for every pair of vertices there is a path joining them. Otherwise the graph is
disconnected.
The \emph{connectivity} (or \emph{vertex-connectivity}) $\kappa(G)$ of a graph $G$ is the minimum number of vertices whose
removal results in a disconnected graph or in the trivial graph.  The \emph{edge-connectivity} $\kappa'(G)$ is defined analogously,
only instead of vertices we remove edges.
A graph is \emph{$k$-connected} if its connectivity is at least $k$ and \emph{$k$-edge-connected} if its edge-connectivity is at least
$k$.
It is almost as simple to check that
the minimal degree $\delta(G)$, the edge-connectivity and \emph{vertex-connectivity} satisfy the following inequality:
\begin{equation*}\begin{array}{ll}
\delta(G)\ge\kappa'(G)\ge\kappa(G).
\end{array}
\end{equation*}
A number of extremal problems related to graph connectivity have been studied
in recent years.
One of the most important task for characterization of $k$-connected graphs is to give  certain operation such that they can be produced from simple $k$-connected graphs by repeatedly  applying  this operation \cite{BB-1}.
This goal has accomplished by Tutte \cite{Tutte} for $3$-connected graphs, by Dirac \cite{GD} and Plummer \cite{MP} for $2$-connected graphs and by Slater \cite{Slater} for $4$-connected graphs.

A graph is said to be \emph{minimally $k$-(edge)-connected}
if it is $k$-(edge)-connected but omitting any of edges the resulting graph is no
longer $k$-(edge)-connected.
Clearly, a $k$-(edge)-connected graph whose every edge
is incident with one vertex of degree $k$ is minimally $k$-(edge)-connected,
especially a $k$-regular and $k$-(edge)-connected graph
is minimally $k$-(edge)-connected.

One of the central problems in this area is to determine
the number of vertices of degree $k$ in a minimally $k$-edge-connected graph.
In 1972, Lick \cite{Lick-2} showed that every minimally $k$-edge-connected finite graph has two vertices of degree $k$ (see also Lemma 13 in \cite{Mader-2}),
which of course is best possible.
But for simple graphs, this was
improved in \cite{Mader-3} as follows:
every minimally $k$-edge-connected finite simple graph
has at least $k+1$ vertices of degree $k$.
It was proved in \cite{Mader-4} that for every $k\notin\{1,3\}$
there exists a $c_k>0$ such
that every minimally $k$-edge-connected finite simple graph $G$
has at least $c_k|G|$ vertices of degree $k$.
The value of the constant $c_k$ was improved in \cite{BB-3} and \cite{Cheng-1},
and a rather good
estimate for $c_k$ was given by Cai \cite{Cheng-2}.
In 1995, Mader \cite{Mader-1} further improved the value $c_k$
and gave the best possible linear bound for $k\equiv 3 (\!\!\mod 4)$.

Another interesting problem is to determine the
maximum number of edges in a minimally $k$-(edge)-connected graph.
Mader \cite{Mader-5} proved that $e(G)\leq kn-\binom{k+1}{2}$ for every minimally $k$-connected graph $G$ of order $n$,
and if $n\geq3k-2$ then $e(G)\leq k(n-k)$, where the equality is uniquely
attained by the complete bipartite graph $K_{k,n-k}$
provided that $k\geq 2$ and $n\geq 3k-1$.
Cai \cite{Cheng-0} proved that
$e(G)\leq \lfloor\frac{(n+k)^2}{8}\rfloor$
for every minimally $k$-connected graph $G$ of order $n<3k-2$.
Mader \cite{Mader-5} also
proved that every minimally $k$-edge-connected graph
on $n$ vertices has at most $k(n-k)$ edges provided $n\geq 3k-2$.
The complete bipartite graph $K_{k,n-k}$ shows that this bound is tight.
Dalmazzo \cite{Dalmazzo} proved that every minimally $k$-edge-connected multidigraph
on $n$ vertices has at most $2k(n-1)$ edges.
In 2005, Berg and Jord\'{a}n \cite{Berg-Jordan} showed that if
multiple edges are not allowed then Dalmazzo's bound can be improved to
$2k(n-k)$ for $n$ sufficiently large. In this paper, we first obtain an extremal result for every subgraph of a minimally $k$-(edge)-connected graph.

\begin{thm}\label{thm-1.0}
Let $G$ be a minimally $k$-(edge)-connected graph and
let $H$ be a subgraph of $G$.
Then $e(H)\leq k(|H|-1)$.
Moreover, if $|H|\geq \frac{1}{2}k(k+5)$,
then $e(H)\leq k(|H|-k)$,
where the equality holds if and only if
$H\cong K_{k,|H|-k}$.
\end{thm}

Let $A(G)$ be the adjacency matrix of a graph $G$.
The largest eigenvalue of $A(G)$ is called the
\emph{spectral radius} of $G$, and denoted by $\rho(G)$.
In classic theory of graph spectra,
many scholars are interested in an extremal problem,
that is, what is the maximal spectral radius
of a family $\mathbb{G}$ of graphs,
where graphs in $\mathbb{G}$ have a common property $P$.
A graph is said to be \emph{$P$-saturated},
if it has the property $P$ but adding an edge between
an arbitrary pair of non-adjacent vertices
results in a graph which does not have the property.
It is known that $A(G)$ is a non-negative matrix,
and adding an edge in $G$ always increases the spectral radius
provided that $G$ is connected.
Therefore, most of spectral extremal problems have saturated extremal graphs,
(see for example, \cite{Cioaba-1,Cioaba-3,Nikiforov,L-P,Kang,N-W-K,Zhai-3,Zhai-1,Zhai-2}).
Particularly, we have the following problem.

\begin{prob}\label{prob-1}
What is the maximal spectral radius among all $n$-vertex saturated graphs
with fixed vertex-connectivity or edge-connectivity?
\end{prob}

Ye, Fan and Wang \cite{Y-F-W}
showed that among all graphs of order $n$ with
vertex (edge)-connectivity $k$,
$K(n-1,k)$ has the maximal spectral radius, where $K(n-1,k)$ is
obtained from the complete graph $K_{n-1}$ by adding a new vertex of degree $k$.
Clearly, $K(n-1,k)$ has the same vertex-connectivity,
edge-connectivity and minimum degree.
Ning, Lu and Wang \cite{N-L-W} proved that
for all graphs of order $n$ with minimum degree $\delta$ and edge connectivity $\kappa'<\delta$,
the maximal spectral radius is attained by joining $\kappa'$ edges between
two disjoint complete graphs $K_{\delta+1}$ and $K_{n-\delta-1}$,
and they also determined the unique extremal graph
with minimum degree $\delta$ and edge-connectivity $\kappa'\in\{0,1,2,3\}$.
Very recently,
Fan, Gu and Lin \cite{DDF-1} determined the unique spectral extremal graph over all $n$-vertex graphs with minimum degree $\delta$ and edge connectivity $\kappa'\in\{4,\ldots,\delta-1\}$.

A graph $G$ is said to be \emph{$P$-critical},
if it admits a property $P$ but $G-e$ does not have for any edge $e\in E(G)$.
Clearly, every minimally $k$-(edge)-connected graph is a
connectivity-critical graph.
Comparing with Problem \ref{prob-1},
the following problem also attracts interest of scholars.

\begin{prob}\label{prob-2}
What is the maximal spectral radius among all $n$-vertex critical graphs
with fixed vertex-connectivity or edge-connectivity?
\end{prob}

Obviously,
every minimally $1$-(edge)-connected graph
is a tree.
Furthermore, it is known that the maximal spectral radius
of a tree is attained uniquely by a star (see \cite{Cvetkovi}).
In 2019, Chen and Guo \cite{CXD} showed that $K_{2,n-2}$ attains the maximal spectral radius among all minimally $2$-connected graphs
and minimally $2$-edge-connected graphs, respectively.
Subsequently,
Fan, Goryainov and Lin \cite{DDF} proved that $K_{3,n-3}$
attains the largest spectral radius over all minimally 3-connected graphs.

Now let $k\geq3$ be a fixed integer and $\alpha=\frac{1}{24k(k+1)}$.
Let $X=(x_1,x_2,\ldots,x_n)^T$
be a non-negative eigenvector with respect to $\rho(G)$.
We may assume that $x_{u^*}=\max_{1\leq i\leq n}x_i$ for some $u^*\in V(G)$.
In this paper,
we prove the following result, which implies that
every minimally $k$-(edge)-connected graph
with large spectral radius contains a certain number of vertices of
high degrees.

\begin{thm}\label{thm-1.1}
Let $G$ be an $n$-vertex minimally $k$-(edge)-connected graph,
where $n\geq \frac{18k}{\alpha^2}$.
If $\rho^2(G)\geq k(n-k)$, then $G$ contains a $k$-vertex subset $L$
such that $x_v\geq (1-\frac{1}{2k})x_{u^*}$ and
$d_G(v)\geq (1-\frac{2}{3k})n$ for each vertex $v\in L$.
\end{thm}

The main result of the paper is the following Max-Min theorem,
which implies that every minimally $k$-(edge)-connected graph with maximal
spectral radius also has maximal number of edges.

\begin{thm}\label{thm-1.2}
For $n\geq \frac{18k}{\alpha^2}$,
the maximal spectral radius of an $n$-vertex minimally $k$-(edge)-connected graph
is attained uniquely by the complete bipartite graph $K_{k,n-k}$.
\end{thm}

Finally, we present the following problem.
\begin{prob}\label{prob-3}
Consider a given property $P$.
Whether an edge-extremal problem on $P$-critical graphs
has a spectral analogue?
\end{prob}

The rest of the paper is organized as follows.
In Section \ref{2}, we give some
structural properties of a minimally $k$-(edge)-connected graph
as well as the proof of Theorem \ref{thm-1.0}. In Section \ref{3},
we use Theorem \ref{thm-1.0}
to show Theorems \ref{thm-1.1} and \ref{thm-1.2}.

\section{Structural properties}\label{2}

Let $G$ be a graph with vertex set $V(G)$ and edge set $E(G)$.
We write $|G|$ for the number of vertices
and $e(G)$ the number of edges in $G$.
For a vertex $v \in V(G)$, let $N_G(v)$ be the neighborhood of $v$.
For $S\subseteq V(G)$, we denote $N_S(v)=N(v)\cap S$ and $d_S(v)=|N_S(v)|$.
The subgraph of $G$ induced by $S$ and $V(G)\setminus S$
are denoted by $G[S]$ and $G-S$, respectively.
Let $e_G(S)$ be the number of edges within $S$, and let
$e_G(S,V(G)\setminus S)$ be the number of edges
between $S$ and $V(G)\setminus S$.
All the subscripts defined here
will be omitted if it is clear from the context.
We start with the following lemma.

\begin{lem}\label{lem-2.1}
Every $k$-(edge)-connected subgraph of a minimally
$k$-(edge)-connected graph is minimally $k$-(edge)-connected.
\end{lem}

\begin{proof}
We first prove that for every subgraph of a minimally
$k$-edge-connected graph, if it is $k$-edge-connected then
it is minimally $k$-edge-connected.
Suppose to the contrary that $H$ is a $k$-edge-connected subgraph of $G$ but it is not minimally $k$-edge-connected. Then there exists an edge, say $u_1u_2$, of $H$ such that $H-u_1u_2$ is also $k$-edge-connected.

Notice that $G$ is a minimally $k$-edge-connected graph.
Hence, $G-u_1u_2$ is $(k-1)$-edge-connected. Thus,
there exists a partition $V(G)=V_1\cup V_2$ such that $u_1\in V_1$, $u_2\in V_2$ and $e(V_1, V_2)=k$.
Now, let $V_i(H)=V(H)\cap V_i$ for $i\in\{1,2\}$.
Clearly,
\begin{equation*}\begin{array}{ll}
e\big(V_1(H),V_2(H)\big)\leq e\big(V_1,V_2\big)=k,
\end{array}
\end{equation*}
and thus $e\big(V_1(H),V_2(H)\big)\leq k-1$ in $H-u_1u_2$,
which contradicts the fact that
$H-u_1u_2$ is $k$-edge-connected.
Therefore, the result follows.

The vertex-connected case of the lemma is
an exercise of Chapter one in \cite{BB-1}.
Hence, we omit its proof here.
\end{proof}

Next, we give the maximal number of edges in every subgraph
of a minimally $k$-edge-connected graph.
Before proceeding, we need two more lemmas due to Mader \cite{Mader-3}.

\begin{lem}[\cite{Mader-3}]\label{lem-2.3}
Let $G$ be a graph of order $n\geq k$.
If $G$ does not contain any $(k+1)$-edge-connected subgraph,
then
\begin{equation*}\begin{array}{ll}
e(G)\leq k(n-k)+\binom{k}{2}.
\end{array}
\end{equation*}
Furthermore, this bound is best possible.
\end{lem}

\begin{lem}[\cite{Mader-3}]\label{lem-2.4}
Let $G$ be a minimally $k$-edge-connected graph of order $n\geq3k$.
Then
\begin{equation*}\begin{array}{ll}
e(G)\leq k(n-k),
\end{array}
\end{equation*}
with equality if and only if $G\cong K_{k,n-k}$.
\end{lem}

\begin{thm}\label{thm-2.1}
Let $G$ be a minimally $k$-edge-connected graph and
let $H$ be a subgraph of $G$.
Then $e(H)\leq k(|H|-1)$.
Moreover, if $|H|\geq\frac12k(k+5)$,
then $e(H)\leq k(|H|-k)$,
where the equality holds if and only if
$H\cong K_{k,|H|-k}$.
\end{thm}

\begin{proof}
Firstly, we will show that $e(H)\leq k(|H|-1)$.
If $|H|<k$,
then $e(H)\leq\frac12|H|(|H|-1)\leq k(|H|-1)$, as desired.
Now assume that $|H|\geq k$.
It suffices to show $e(H)\leq k(|H|-\frac{k+1}{2})$.
By Lemma \ref{lem-2.1}, every $k$-edge-connected subgraph of $G$
is minimally $k$-edge-connected, and thus has edge-connectivity $k$.
Hence, $G$ contains no $(k+1)$-edge-connected subgraphs.
By Lemma \ref{lem-2.3}, we have
$e(H)\leq k(|H|-k)+\binom{k}{2}=k(|H|-\frac{k+1}{2})$,
as required.

In the following, we prove that
$e(H)\leq k(|H|-k)$ for $|H|\geq\frac12k(k+5)$.
The proof should be distinguished into two cases.

\textbf{Case 1: $H$ contains no $k$-edge-connected subgraphs.}
By Lemma \ref{lem-2.3}, we know that
$e(H)\leq (k-1)(|H|-\frac{k}{2})$.
Note that $|H|\geq\frac12k(k+5)$.
It is easy to see that
$(k-1)(|H|-\frac{k}{2})<k(|H|-k)$,
and the result follows.

\textbf{Case 2: $H$ contains $k$-edge-connected subgraphs.}
Let $H_0$ be a maximal $k$-edge-connected subgraph of $H$.
Then $H_0$ is a vertex-induced subgraph with $|H_0|\geq k+1$.
If $H=H_0$, then by Lemma \ref{lem-2.1},
$H$ is minimally $k$-edge-connected.
Since $|H|\geq\frac12k(k+5)\geq3k$, by Lemma \ref{lem-2.4}
we have $e(H)\leq k(|H|-k),$
with equality if and only if $H\cong K_{k,|H|-k}$.

Now we may assume that $H_0$ is a proper induced subgraph of $H$.
Then $\kappa'(H)\leq k-1$, and thus
we can find a partition $V(H)=V_0\cup V_1$ such that
$e(H)\leq e(V_0)+e(V_1)+(k-1).$
One can observe that $H_0$ is a subgraph of $H[V_0]$ or $H[V_1]$
(otherwise, write $U_i=V(H_0)\cap V_i$ for $i\in\{0,1\}$,
then $e(U_0,U_1)\geq k$ as $H_0$ is $k$-edge-connected,
consequently, $e(V_0,V_1)\geq k$, a contradiction).
For $i\in\{0,1\}$, if $\kappa'(H[V_i])\leq k-1$ and $|V_i|\geq2$, then
we can find a partition $V_i=V_i'\cup V_i''$ such that
$e(V_i)\leq e(V_i')+e(V_i'')+(k-1)$.
Similarly, every $k$-edge-connected subgraph of $H[V_i]$
can only be a subgraph of $H[V_i']$ or $H[V_i'']$.

By a series of above iterate operations (say $s$ steps),
we can obtain a partition $V(H)=\cup_{i=0}^sV_i$
satisfying that
\begin{equation}\label{eq-1.1}
\begin{array}{ll}
e(H)\leq\sum\limits_{i=0}^se(V_i)+(k-1)s
\end{array}
\end{equation}
and every $H[V_i]$ is either $k$-edge-connected or a single vertex.
Recall that $G$ contains no $(k+1)$-edge-connected subgraphs.
If $H[V_i]$ is $k$-edge-connected,
then $|V_i|\geq k+1$ and
$e(V_i)\leq k(|V_i|-\frac{k+1}{2})$ by Lemma \ref{lem-2.3}.
Let $S_1=\{i\mid |V_i|=1\}$ and $S_2=\{0,\ldots,s\}\setminus S_1$.
Then $s=|S_1|+|S_2|-1$ and $|H|=\sum_{i\in S_2}|V_i|+|S_1|.$
In view of (\ref{eq-1.1}), we have
\begin{equation}\label{eq-1.2}
\begin{array}{ll}
e(H)&\leq\sum\limits_{i\in S_2}k(|V_i|-\frac{k+1}{2})+(k-1)(|S_1|+|S_2|-1)\vspace{0.2cm}\\
&=k|H|-\frac{1}{2}(k^2-k+2)|S_2|-|S_1|-(k-1).
\end{array}
\end{equation}

If $|S_2|\geq2$, then $\frac{1}{2}(k^2-k+2)|S_2|+(k-1)>k^2$,
and so $e(H)<k(|H|-k)$, as desired.
Now assume that $|S_2|=1$ (say $S_2=\{0\}$ and $H[V_0]=H_0$).
Then $S_1\ne \varnothing$ as $H_0$ is a proper induced subgraph of $H$.
By Lemma \ref{lem-2.1},
$H_0$ is minimally $k$-edge-connected.
If $|H_0|\geq3k$, then by Lemma \ref{lem-2.4},
we have $e(H_0)\leq k(|V_0|-k)$.
Combining (\ref{eq-1.1}), we obtain
$e(H)\leq k(|V_0|-k)+(k-1)|S_1|=k(|H|-k)-|S_1|.$
The result follows.
If $|H_0|<3k$, then
$|S_1|=|H|-|H_0|>\frac12k(k-1)$,
and by (\ref{eq-1.2}) we have
$e(H)\leq k|H|-\frac{1}{2}(k^2-k+2)-|S_1|-(k-1)<k(|H|-k)$.
This completes the proof.
\end{proof}

Now we give a vertex-connected version of
Theorem \ref{thm-2.1}, which will be proved by
a different approach.

\begin{lem}[\cite{BB-1}]\label{lem-2.2}
Let $G$ be a minimally $k$-connected graph and let $S$ be the set of vertices of degree $k$ in $G$. Then $G-S$ is empty or a forest.
\end{lem}

Recall that $e(G)\leq k(n-k)$
for $n\geq3k-2$ and every $n$-vertex minimally $k$-connected graph $G$.
We also want to know the maximal number of edges in every subgraph
of a minimally $k$-connected graph.

\begin{thm}\label{thm-2.2}
Let $G$ be a minimally $k$-connected graph and
let $H$ be a subgraph of $G$.
Then $e(H)\leq k(|H|-1)$.
Moreover, if $|H|\geq 5k-4$,
then $e(H)\leq k(|H|-k)$,
where the equality holds if and only if
$H\cong K_{k, |H|-k}$.
\end{thm}

\begin{proof}
Firstly, we show $e(H)\leq k(|H|-1)$.
We partition $V(H)$ into two parts: $V(H)=V_1\cup V_2$,
where $V_1$ is the set of vertices of degree $k$ in $G$.
If $|V_2|=0$, then
$e(H)\leq\frac{k|H|}{2}\leq k(|H|-1)$, as desired.
If $|V_2|\geq 1$,
from Lemma \ref{lem-2.2} we know that
$G[V_2]$ is a forest, and so $e(V_2)\leq |V_2|-1$.
Thus, we can get an upper bound of $e(H)$ as below:
\begin{equation}\label{eq-1.0}
\begin{array}{ll}
e(H)=e(V_1)+e(V_1,V_2)+e(V_2)\leq k|V_1|+(|V_2|-1),
\end{array}
\end{equation}
where the equality holds
if and only if $G[V_2]$ is a tree and
$N_G(v)\subseteq V_2$ for each $v\in V_1$.
It is clear that $k|V_1|+|V_2|-1\leq k(|V_1|+|V_2|-1)$,
and hence $e(H)\leq k(|H|-1)$.

Next, we shall distinguish three cases to
show $e(H)\leq k(|H|-k)$ for $|H|\geq 5k-4$.
If $k=1$, then $G$ is a tree. Clearly, the result holds.  In the following, we may assume $k\leq 2$.

\textbf{Case 1: $|V_2|\geq k+1$.}
From (\ref{eq-1.0}) we have
\begin{equation*}\begin{array}{ll}
e(H)\leq k|V_1|+|V_2|-1<k(|V_1|+|V_2|-k)=k(|H|-k).
\end{array}
\end{equation*}
The result follows.

\textbf{Case 2: $|V_2|=k$.} Then $|V_1|\geq 4(k-1)$.
If $e(V_2)=0$, then by (\ref{eq-1.0}),
we have $e(H)\leq k|V_1|=k(|V_1|+|V_2|-k)=k(|H|-k)$,
with equality if and only if $H\cong K_{k,|H|-k}$.

Now, assume that $e(V_2)\geq1$,
and let $V_1'=\{v\in V_1\mid N_G(v)=V_2\}$.
Then $K_{|V_1'|,|V_2|}\subseteq G[V_1'\cup V_2]$.
We will see that $|V_1'|\leq k-1$.
Otherwise, if $|V_1'|\geq k$,
then $G[V_1'\cup V_2]$ is $k$-connected.
By Lemma \ref{lem-2.1}, $G[V_1'\cup V_2]$ is
minimally $k$-connected, which implies that
$G[V_1'\cup V_2]\cong K_{|V_1'|, |V_2|}$ and so $e(V_2)=0$, a contradiction.
Hence, $|V_1'|\leq k-1$.

On the other hand, let $V_1''=V_1\setminus V_1',$
then
\begin{equation*}\begin{array}{ll}
e\big(V_1''\big)+e\big(V_1'',V_2\big)
\leq\big(|V_2|-1\big)|V_1''|+\frac12|V_1''|
=\big(k-\frac12\big)|V_1''|.
\end{array}
\end{equation*}
Since $|V_1'|\leq k-1$ and $|V_1'|+|V_1''|=|V_1|$, we further obtain
\begin{equation*}\begin{array}{ll}
e\big(V_1\big)+e\big(V_1,V_2\big)\leq k|V_1'|
+e\big(V_1''\big)+e\big(V_1'',V_2\big)
\leq\big(k-\frac{1}{2}\big)|V_1|+\frac12\big(k-1\big).
\end{array}
\end{equation*}
Recall that $|V_1|\geq4(k-1)$ and $e(V_2)\leq k-1$.
Thus we also have
\begin{equation*}\begin{array}{ll}
e\big(H\big)
\leq\big(k-\frac{1}{2}\big)|V_1|+\frac32\big(k-1\big)
<k|V_1|=k\big(|H|-k\big).
\end{array}
\end{equation*}

\textbf{Case 3: $|V_2|\leq k-1$.} Then $|V_1|\geq4k-3$.
Let $|V_1|=x$ and $|V_2|=y$.
Then
\begin{equation*}\begin{array}{ll}
e(H)&=e(V_1,V_2)+e(V_1)+e(V_2)\vspace{0.2cm}\\
&\leq|V_1||V_2|+\frac{1}{2}|V_1|(k-|V_2|)+(|V_2|-1)\vspace{0.2cm}\\
&=\frac{1}{2}xy+\frac{1}{2}kx+(y-1),
\end{array}
\end{equation*}
Notice that $k(|H|-k)=k(x+y-k)$.
Let
\begin{equation*}\begin{array}{ll}
f(x,y)=\frac{1}{2}xy+\frac{1}{2}kx+(y-1)-k(x+y-k).
\end{array}
\end{equation*}
It suffices to show $f(x,y)<0$ for $x\geq 4k-3$ and $y\leq k-1$.
Note that $\frac{\partial f(x,y)}{\partial x}=\frac{1}{2}(y-k)<0$ and $\frac{\partial f(x,y)}{\partial y}=\frac{1}{2}x+1-k>0$.
Hence, $f(x,y)$ is decreasing with respect to $x$
and increasing with respect to $y$.
Therefore,
$f(x,y)\mid_{max}=f(4k-3,k-1)=-\frac12,$
as desired.
\end{proof}

Observe that $\frac12k(k+5)\geq 5k-4$ for every positive integer $k$.
Combining Theorems \ref{thm-2.1} and \ref{thm-2.2},
we immediately obtain Theorem \ref{thm-1.0}.

\section{Spectral extremal results}\label{3}

Let $G$ be a minimally $k$-(edge)-connected graph
of order $n$.
By Perron-Frobenious theorem, there exists a
positive unit eigenvector with respect to $\rho(G)$,
which is called the \emph{Perron vector} of $G$.
Let $X=(x_1,x_2,\ldots,x_n)^T$ be the Perron vector
with coordinate $x_{u^*}=\max\{x_i\mid i\in V(G)\}$.
In this section, we first show Theorem \ref{thm-1.1},
that is, if $\rho^2(G)\geq k(n-k)$, then $G$ contains a $k$-vertex subset
$L$ such that $x_v\geq (1-\frac{1}{2k})x_{u^*}$ and
$d(v)\geq (1-\frac{2}{3k})n$ for each vertex $v\in L$.
Before proceeding, we define three subsets of $V(G)$.
\begin{equation*}
\begin{array}{ll}
L_{\alpha}=\{v\in V(G)\mid x_v>\alpha x_{u^*}\},& \mbox{where $0<\alpha\leq \frac{1}{24k(k+1)}$};\vspace{0.2cm}\\
L_{\beta}=\{v\in V(G)\mid x_v>\beta x_{u^*}\},& \mbox{where $\frac53\alpha\leq\beta\leq \frac{1}{6k^2}$};\vspace{0.2cm}\\
L_\gamma=\{v\in V(G)\mid x_v\geq\gamma x_{u^*}\},&
\mbox{where $\frac{1}{2k}\leq\gamma\leq1$}.
\end{array}
\end{equation*}
Clearly, $L_\gamma\subseteq L_{\beta}\subseteq L_{\alpha}$.
In the following,
assume that $k\geq3$ and $n\geq\frac{18k}{\alpha^2}$.
We shall prove some lemmas on these three subsets.

\begin{lem}\label{lem-3.1}
$|L_{\alpha}|<\sqrt{4kn}$.
\end{lem}

\begin{proof}
For every $v\in L_{\alpha}$,
we have $\rho x_{v}=\sum_{v\in N(v)}x_v$, and thus
\begin{equation}\label{eq-2.1}
\begin{array}{ll}
\rho x_{v}=\!\!\!\sum\limits_{v\in N(v)\cap L_{\alpha}}\!\!\!\!x_v
+\!\!\!\sum\limits_{v\in N(v)\setminus L_{\alpha}}\!\!\!x_v
\leq\Big(d_{L_{\alpha}}(v)+\alpha\cdot d_{V(G)\setminus L_{\alpha}}(v)\Big)x_{u^*}.
\end{array}
\end{equation}
Since $\rho x_{v}\geq\sqrt{k(n-k)}\alpha x_{u^*}$ for $v\in L_{\alpha}$,
from (\ref{eq-2.1}) we have
\begin{equation}\label{eq-2.2}
\sqrt{k(n-k)} \alpha \leq d_{L_{\alpha}}(v)+\alpha\cdot
d_{V(G)\setminus L_{\alpha}}(v).
\end{equation}
Summing (\ref{eq-2.2}) over all $v\in L_{\alpha}$, we have
\begin{equation}\label{eq-2.3}
|L_{\alpha}| \sqrt{k(n-k)}\alpha\leq 2e(L_{\alpha})
+\alpha\cdot e(L_{\alpha},V(G)\setminus L_{\alpha}).
\end{equation}
By Theorem \ref{thm-1.0}, we have
$e(L_{\alpha})\leq k|L_{\alpha}|$
and $e(L_{\alpha}, V(G)\setminus L_{\alpha})\leq e(G)\leq k(n-k)$.
Combining (\ref{eq-2.3}), we get that
\begin{equation}\label{eq-2.4}
|L_{\alpha}|\sqrt{k(n-k)}\leq\frac{2k}{\alpha}|L_{\alpha}|+k(n-k).
\end{equation}
Since $n\geq\frac{18k}{\alpha^2}$, we have $n-k>\frac{16k}{\alpha^2}$,
and hence $\frac{2k}{\alpha}<\frac{1}{2}\sqrt{k(n-k)}$.
Combining (\ref{eq-2.4}), we obtain that $|L_{\alpha}|<2\sqrt{k(n-k)}$,
and thus $|L_{\alpha}|<\sqrt{4kn}$, as desired.
\end{proof}

For a vertex $v\in V(G)$,
let $N[v]=N(v)\cup \{v\}$
and $N^2(v)$ denote the set of vertices at distance two
from $v$.

\begin{lem}\label{lem-3.2}
$|L_{\beta}|<\frac{12k}{\alpha}$.
\end{lem}

\begin{proof}
We proceed the proof by contradiction.
Suppose that $|L_{\beta}|\geq\frac{12k}{\alpha}$.
Recall that $L_{\beta}\subseteq L_\alpha$ and $\alpha\leq\frac1{24k(k+1)}$.
Then $|L_\alpha|\geq\frac{12k}{\alpha}\geq
\max\{5k-4,\frac12k(k+5)\}$.
We first prove that $d(v)>\frac{\alpha}{12}n+k$
for each vertex $v\in L_\beta$.

By Theorem \ref{thm-1.0}, we get that
$e(G)\leq kn$,
$e(N[v])\leq k(|N[v]|-1)=kd(v)$ and
$e(N(v)\cup L_{\alpha})\leq k(d(v)+|L_\alpha|-k)$.
Since $v\in L_\beta$, we can easily see that $v\in L_\alpha$.
Let $S=N(v)\cup(L_{\alpha}\setminus\{v\})$.
Then $e(S)=e(N(v)\cup L_{\alpha})-d(v)\leq(k-1)d(v)+k|L_\alpha|-k^2$,
where $|L_{\alpha}|<\sqrt{4kn}<\frac{\alpha}2n$
by Lemma \ref{lem-3.1} and the assumption that $n\geq\frac{18k}{\alpha^2}$.

It is easy to see that
\begin{equation*}
\begin{array}{ll}
d(v)x_{v}+\!\!\sum\limits_{v\in N(v)}\!\!d_{N(v)}(u)x_u
\leq\Big(d(v)+2e(N(v)\Big)x_{u^*}=\Big(e(N[v])+e(N(v))\Big)x_{u^*}.
\end{array}
\end{equation*}
Note that $S=N(v)\cup(L_{\alpha}\setminus\{v\})$.
Then $e\big(N^2(v)\cap L_{\alpha},N(v)\big)\leq e(S)-e\big(N(v)\big)$
and
\begin{equation*}
\begin{array}{ll}
\sum\limits_{u\in N^2(v)}\!\!\!d_{N(v)}(u)x_u=
\!\!\!\!\sum\limits_{u\in N^2(v)\cap L_{\alpha}}\!\!\!\!d_{N(v)}(u)x_u
+\!\!\!\!\sum\limits_{u\in N^2(v)\setminus L_{\alpha}}\!\!\!\!d_{N(v)}(u)x_u
\leq\Big(e(S)-e(N(v))
+\alpha\cdot e(G)\Big)x_{u^*}.
\end{array}
\end{equation*}
Combining the above two inequalities, we obtain
\begin{equation*}
\begin{array}{ll}
\rho^2x_{v}&=d(v)x_{v}+\sum\limits_{v\in N(v)}d_{N(v)}(u)x_u+\sum\limits_{u\in N^2(v)}d_{N(v)}(u)x_u \vspace{0.2cm}\\
&\leq\Big(e(N[v])+e(S)+\alpha\cdot e(G)\Big)x_{u^*}.\vspace{0.2cm}\\
&\leq\Big((2k-1)d(v)+\frac{3\alpha}2kn-k^2\Big)x_{u^*}.
\end{array}
\end{equation*}
Notice that $\frac53\alpha\leq\beta<1$ and
$\rho^2 x_{v}\geq k(n-k)\beta x_{u^*}>(\beta kn-k^2)x_{u^*}$
for each vertex $v\in L_\beta$.
In view of the above inequality, we have
$(\beta-\frac32\alpha)kn<(2k-1)d(v),$
which yields that
$d(v)>\frac{k}{2k-1}(\beta-\frac32\alpha)n>\frac{\alpha}{12}n+k$
for each vertex $v\in L_\beta$.

By Theorem \ref{thm-1.0},
we also have $e(L_{\beta})\leq k|L_{\beta}|$.
Observe that
$\sum_{u\in V(G)\setminus L_{\beta}}d(u)
\geq e(L_{\beta},V(G)\setminus L_{\beta})
=\sum_{v\in L_{\beta}}d(v)-2e(L_{\beta}).$
Therefore,
\begin{equation*}
\begin{array}{ll}
2e(G)=\sum\limits_{v\in L_{\beta}}d(v)
+\!\!\sum\limits_{u\in V(G)\setminus L_{\beta}}\!\!d(u)
\geq2\sum\limits_{v\in L_{\beta}}d(v)-2e(L_{\beta})
>|L_{\beta}|\frac{\alpha}{6}n.
\end{array}
\end{equation*}
Combining $e(G)\leq kn$,
we obtain $|L_\beta|<\frac{12k}{\alpha}$.
This completes the proof.
\end{proof}

\begin{lem}\label{lem-3.3}
$d(v)>\big(\gamma-\frac{1}{6k}\big)n$ for each $v\in L_\gamma$.
\end{lem}

\begin{proof}
Suppose to the contrary that there exists a vertex $v_0\in L_\gamma$
with $d(v_0)\leq(\gamma-\frac{1}{6k})n$.
We may assume that $x_{v_0}=\gamma_0x_{u^*}$.
By the definition of $L_\gamma$,
we know that $\frac{1}{2k}\leq\gamma\leq \gamma_0\leq1$,
and thus $d(v_0)\leq(\gamma_0-\frac{1}{6k})n$.

Set $R=N(v_0)\cup N^2(v_0)$.
Then $x_v\leq\beta x_{u^*}$ for each $v\in R\setminus L_\beta$.
Therefore,
\begin{equation}\label{eq-2.5}
\begin{array}{ll}
\rho^2x_{v_0}&=d(v_0)x_{v_0}+\sum\limits_{v\in R}d_{N(v_0)}(v)x_v \vspace{0.2cm}\\
&=d(v_0)x_{v_0}+\sum\limits_{v\in R\setminus L_{\beta}}d_{N(v_0)}(v)x_v+
\sum\limits_{v\in R\cap L_{\beta}}d_{N(v_0)}(v)x_v\vspace{0.2cm}\\
&\leq \Big(\gamma_0d(v_0)+\beta\sum\limits_{v\in R\setminus L_{\beta}}d_{N(v_0)}(v)+\sum\limits_{v\in R\cap L_{\beta}}d_{N(v_0)}(v)\Big)x_{u^*}.
\end{array}
\end{equation}
Since $N(v_0)\subseteq R$, we can see that
\begin{equation}\label{eq-2.6}
\begin{array}{ll}
\sum\limits_{v\in R\setminus L_{\beta}}d_{N(v_0)}(v)
\leq\sum\limits_{v\in R}d_{R}(v)=2e(R)\leq2e(G)\leq2kn.
\end{array}
\end{equation}
Observe that $R\cap L_{\beta}\subseteq L_{\beta}\setminus \{v_0\}$.
We also have
\begin{equation}\label{eq-2.7}
\begin{array}{ll}
\sum\limits_{v\in R\cap L_{\beta}}d_{N(v_0)}(v)
&\leq\sum\limits_{v\in L_{\beta}\setminus \{v_0\}}d_{N(v_0)\cap L_{\beta}}(v)+\sum\limits_{v\in L_{\beta}\setminus \{v_0\}}d_{N(v_0)\setminus L_{\beta}}(v)\vspace{0.2cm}\\
&\leq 2e\big(L_{\beta}\big)+e\big(L_{\beta},N(v_0)\setminus L_{\beta}\big)
-|N(v_0)\setminus L_{\beta}|.
\end{array}
\end{equation}
Furthermore,
$e\big(L_{\beta},N(v_0)\setminus L_{\beta}\big)
\leq e\big(L_{\beta}\cup N(v_0)\big)-e\big(L_{\beta}\big).$
Notice that $e(L_\beta)\leq k|L_\beta|$ and
$e\big(L_{\beta}\cup N(v_0)\big)\leq k\big(|L_\beta|+d(v_0)\big)$.
Combining (\ref{eq-2.7}),
we obtain
\begin{equation}\label{eq-2.8}
\begin{array}{ll}
\sum\limits_{v\in R\cap L_{\beta}}d_{N(v_0)}(v)
&\leq e\big(L_{\beta}\cup N(v_0)\big)-|N(v_0)\setminus L_{\beta}|+e(L_{\beta})\\
&\leq(k-1)d(v_0)+(k+1)|L_{\beta}|+e(L_{\beta})\vspace{0.2cm}\\
&\leq(k-1)d(v_0)+(2k+1)|L_{\beta}|.
\end{array}
\end{equation}
Substituting (\ref{eq-2.6}) and (\ref{eq-2.8}) into (\ref{eq-2.5}),
we get that
\begin{equation}\label{eq-2.9}
\begin{array}{ll}
\rho^2x_{v_0}
&\leq\Big(\gamma_0d(v_0) +2k\beta n+(k-1)d(v_0)+(2k+1)|L_{\beta}|\Big)x_{u^*} \vspace{0.2cm}\\
&=\Big((\gamma_0+k-1)d(v_0)+2k\beta n+(2k+1)|L_{\beta}|\Big)x_{u^*}.
\end{array}
\end{equation}
Since $n\geq\frac{18k}{\alpha^2}$ and $\alpha<\frac{1}{24k^2}$,
we have $\frac{12k}{\alpha}\leq\frac{2}{3}\alpha n<\frac{n}{(6k)^2}$.
Moreover, by Lemma \ref{lem-3.2},
we have $|L_{\beta}|<\frac{12k}{\alpha}$.
Thus, we can check that
$(2k+1)|L_{\beta}|<\frac n{6k}-k^2\leq\frac n{6k}-k^2\gamma_0$.
Recall that $\rho^2x_{v_0}\geq k(n-k)\gamma_0x_{u^*}$
and $d(v_0)\leq (\gamma_0-\frac{1}{6k})n$.
Combining (\ref{eq-2.9}),
we obtain that
\begin{equation*}
\begin{array}{ll}
k(n-k)\gamma_0<
\big(\gamma_0+k-1\big)\big(\gamma_0-\frac{1}{6k}\big)n
+2k\beta n+\frac n{6k}-k^2\gamma_0,
\end{array}
\end{equation*}
which gives
$k\gamma_0<(\gamma_0+k-1)(\gamma_0-\frac{1}{6k})+2k\beta+\frac{1}{6k}.$
Recall that $\beta\leq \frac1{6k^2}$.
It follows that
\begin{equation}\label{eq-2.11}
\begin{array}{ll}
\big(\gamma_0-1\big)\big(\gamma_0-\frac{1}{6k}\big)
>\frac{k-1}{6k}-2k\beta\geq\frac{k-3}{6k}\geq0.
\end{array}
\end{equation}
Now let $f(\gamma)=(\gamma-1\big)\big(\gamma-\frac{1}{6k})$,
where $\frac{1}{2k}\leq\gamma\leq1$.
Obviously, $f(\gamma)|_{\max}=f(1)=0$,
which contradicts (\ref{eq-2.11}).
The proof is completed.
\end{proof}

Recall that $L_\gamma=\{u\in V(G)\mid x_u\geq\gamma x_{u^*}\}$,
where $\frac1{2k}\leq \gamma\leq1$.
Let $\gamma_0:=\frac1{2k}$. Clearly, $L_{1-\gamma_0} \subseteq L_{\gamma_0}$.
We will see that every vertex $u\in L_{\gamma_0}$ has
a larger value $x_u$.

\begin{lem}\label{lem-3.4}
$L_{\gamma_0}=L_{1-\gamma_0}$.
\end{lem}

\begin{proof}
Suppose to the contrary that there exists a vertex
$u_0\in L_{\gamma_0}\setminus L_{1-\gamma_0}.$
Assume that $x_{u_0}=\gamma x_{u^*}$. Then
$\gamma_0\leq\gamma<1-\gamma_0$.
Setting $R=N[u^*]\cup N^2(u^*)$, we have
\begin{equation}\label{eq-2.12}
\begin{array}{ll}
\rho^2x_{u^*}
=\sum\limits_{u\in R}d_{N(u^*)}(u)x_u
=\sum\limits_{u\in R\setminus L_{\beta}}d_{N(u^*)}(u)x_u+
\sum\limits_{u\in R\cap L_{\beta}}d_{N(u^*)}(u)x_u.
\end{array}
\end{equation}
Recall that $e(G)\leq kn$ and
$x_u\leq\beta x_{u^*}$ for each $u\in R\setminus L_{\beta}$.
Then
\begin{equation}\label{eq-2.13}
\begin{array}{ll}
\sum\limits_{u\in R\setminus L_{\beta}}d_{N(u^*)}(u)x_u\leq
\sum\limits_{u\in R}d_{R}(u)\beta x_{u^*}\leq 2e(G)\beta x_{u^*}
\leq 2\beta knx_{u^*}.
\end{array}
\end{equation}
On the other hand,
since $u_0\in L_{\gamma_0}$ and $L_{\gamma_0}\subseteq L_\beta$,
we have $u_0\in L_\beta,$
and thus
\begin{equation}\label{eq-2.14}
\begin{array}{ll}
\sum\limits_{u\in R\cap L_{\beta}}d_{N(u^*)}(u)x_u
\leq\sum\limits_{u\in L_{\beta}}d_{N(u^*)}(u)x_{u^*}+
d_{N(u^*)}(u_0)(x_{u_0}-x_{u^*}),
\end{array}
\end{equation}
where $x_{u_0}-x_{u^*}=(\gamma-1)x_{u^*}$ and
\begin{equation*}
\begin{array}{ll}
\sum\limits_{u\in L_{\beta}}d_{N(u^*)}(u)
&=\sum\limits_{u\in L_{\beta}}d_{N(u^*)\setminus L_{\beta}}(u)+\sum\limits_{u\in L_{\beta}}d_{N(u^*)\cap L_{\beta}}(u)\vspace{0.2cm}\\
&\leq e(L_{\beta}, N(u^*)\setminus L_{\beta})+2e(L_{\beta})\vspace{0.2cm}\\
&\leq e(G)+e(L_{\beta}).
\end{array}
\end{equation*}
Recall that $e(G)\leq k(n-k)$ and
$e(L_\beta)\leq k|L_\beta|<\frac{12}\alpha k^2.$
Consequently,
$\sum_{u\in L_{\beta}}d_{N(u^*)}(u)\leq k(n-k)+\frac{12}{\alpha}k^2.$
Combining (\ref{eq-2.12})-(\ref{eq-2.14}),
we obtain
\begin{equation}\label{eq-2.15}
\begin{array}{ll}
\rho^2x_{u^*}
\leq \Big(2\beta kn+k(n-k)+\frac{12}{\alpha}k^2+(\gamma-1)d_{N(u^*)}(u_0)\Big)x_{u^*}.
\end{array}
\end{equation}
By Lemma \ref{lem-3.3},  
we have $d(u^*)\geq (1-\frac{1}{6k})n$
and $d(u_0)\geq (\gamma-\frac{1}{6k})n$.
Thus, $|V(G)\setminus N(u^*)|\leq\frac{n}{6k}$ and 
$d_{N(u^*)}(u_0)\geq (\gamma-\frac{1}{3k})n$.
Notice that $\rho^2\geq k(n-k)$.
It follows from (\ref{eq-2.15}) that
\begin{equation*}
\begin{array}{ll}
\big(\gamma-1\big)\big(\gamma-\frac{1}{3k}\big)n
\geq-\big(2\beta kn+\frac{12}{\alpha}k^2\big).
\end{array}
\end{equation*}
Recall that $\alpha\leq\frac1{24k(k+1)}$, $\beta\geq\frac53\alpha$
and $n\geq\frac{18k}{\alpha^2}.$
Then $\frac{12}{\alpha}k^2\leq \frac23\alpha kn.$
Now choose $\beta=\frac53\alpha$.
Then we have $2\beta kn+\frac{12}{\alpha}k^2\leq4\alpha kn$,
and hence
$(\gamma-1)(\gamma-\frac{1}{3k})\geq-4\alpha k\geq-\frac1{6(k+1)}$.
Let $f(\gamma)=(\gamma-1)(\gamma-\frac{1}{3k})$,
where $\gamma_0\leq\gamma\leq1-\gamma_0$ and $\gamma_0=\frac{1}{2k}$.
Obviously,
$f(\gamma)|_{\max}=f(\gamma_0)=-\frac{2k-1}{12k^2}<-\frac1{6(k+1)}$
for $k\geq3$,
a contradiction.
\end{proof}

Having the above lemmas,
we now give the proof of Theorem \ref{thm-1.1}.

\begin{proof}
Choose $L=L_{\gamma_0}$ in Theorem \ref{thm-1.1}.
Given an arbitrary vertex $v\in L$.
By Lemma \ref{lem-3.4}, we have
$v\in L_{1-\gamma_0}$,
and thus $x_v\geq(1-\gamma_0)x_{u^*}=(1-\frac1{2k})x_{u^*}.$
Furthermore,
by Lemma \ref{lem-3.3} we have
$d(v)\geq (1-\gamma_0-\frac1{6k})n=(1-\frac{2}{3k})n$.

In the following, it remains to show $|L|=k$.
Firstly, suppose that $|L|\geq k+1$.
Taking $v_1,v_2,\ldots,v_{k+1}$ from $L$, we have
\begin{equation*}
\begin{array}{ll}
\big|\bigcap\limits_{i=1}^{k+1}N(v_i)\big|
\geq\sum\limits_{i=1}^{k+1}\big|N(v_i)\big|-k\big|\bigcup\limits_{i=1}^{k+1}N(v_i)\big|
\geq\big(k+1\big)\big(1-\frac{2}{3k}\big)n-kn=\frac{k-2}{3k}n\geq k+1.
\end{array}
\end{equation*}
Thus, $G$ contains a copy of $K_{k+1,k+1}$,
which is clearly a $(k+1)$-(edge)-connected subgraph.
However, by Lemma \ref{lem-2.1},
every $k$-(edge)-connected subgraph of $G$ is minimally $k$-(edge)-connected,
which implies that $G$ contains no any $(k+1)$-(edge)-connected subgraph.
We get a contradiction. Therefore, $|L|\leq k$.

Finally, suppose that $|L|\leq k-1$.
Since $L=L_{\gamma_0},$
we have $x_v<\gamma_0x_{u^*}=\frac{1}{2k}x_{u^*}$
for every $v\in V(G)\setminus L$.
Setting $R=N[u^*]\cup N^2(u^*)$, we have
\begin{equation}\label{eq-2.16}
\begin{array}{ll}
\rho^2x_{u^*}
=\sum\limits_{u\in R}d_{N(u^*)}(u)x_u
\leq\Big(\sum\limits_{u\in R\cap L}d_{N(u^*)}(u)
+\frac{1}{2k}\sum\limits_{u\in R\setminus L}d_{N(u^*)}(u)\Big)x_{u^*},
\end{array}
\end{equation}
Let $E_0$ be the set of edges
incident to vertices of $L$.
Then, every edge in $E_0$
can not be counted twice in
$\sum_{u\in R\setminus L}d_{N(u^*)}(u)$.
Moreover, it is easy to see that $u^*\in L$
and every edge incident to $u^*$
can not be counted in
$\sum_{u\in R\setminus L}d_{N(u^*)}(u)$.
Consequently,
$\sum_{u\in R\setminus L}d_{N(u^*)}(u)
\leq2e(G)-d(u^*)-|E_0|.$
Note that $e(G)\leq kn$ and
\begin{equation*}
\begin{array}{ll}
d(u^*)+|E_0|=d(u^*)+\sum\limits_{v\in L}d(v)-e(L)
\geq\big(|L|+1\big)\big(1-\frac{2}{3k}\big)n-\frac12k^2
\geq|L|n.
\end{array}
\end{equation*}
It follows that
$\sum_{u\in R\setminus L}d_{N(u^*)}(u)\leq(2k-|L|)n.$
Observe that
$\sum_{u\in R\cap L}d_{N(u^*)}(u)\leq|L|n$.
Combining (\ref{eq-2.16}) and $|L|\leq k-1$, we obtain
\begin{equation*}
\begin{array}{ll}
\rho^2\leq|L|n+\frac{1}{2k}\big(2k-|L|\big)n
\leq(k-1)n+\frac{(k+1)}{2k}n=kn-\frac{k-1}{2k}n,
\end{array}
\end{equation*}
which contradicts $\rho^2\geq k(n-k).$
Therefore, $|L|=k$. This completes the proof.
\end{proof}

At the end of this section,
we give the proof of Theorem \ref{thm-1.2}.

\begin{proof}
Let $G^*$ be a graph with maximal spectral radius
over all minimally $k$-(edge)-connected graphs of order $n$,
where $n\geq \frac{18k}{\alpha^2}$ and $\alpha=\frac{1}{24k(k+1)}$.
Since $K_{k,n-k}$ is also minimally $k$-(edge)-connected,
we have $\rho^2(G^*)\geq\rho^2(K_{k,n-k})=k(n-k).$
Furthermore, by Theorem \ref{thm-1.1},
$G^*$ contains a $k$-vertex subset $L$ such that
$x_{v}\geq(1-\frac{1}{2k})x_{u^*}$ and $d(v)\geq(1-\frac{2}{3k})n$
for each vertex $v\in L$, where $L=L_{\frac{1}{2k}}.$

Denote by $V$ the common neighbourhood of vertices in $L$,
and let $U=V(G^*)\setminus(L\cup V)$.
Since $|L|=k$ and every vertex in $L$
has at most $\frac{2}{3k}n$ non-neighbors,
we can see that
\begin{equation*}
\begin{array}{ll}
|L\cup V|\geq n-k\cdot\frac{2}{3k}n=\frac n3
>\frac12k(k+5).
\end{array}
\end{equation*}

The key point is to show that $U=\varnothing$.
Suppose to the contrary that
$|U|=t\neq0$.
By Theorem \ref{thm-1.0},
we have $e(G)\leq k(n-k)=k(|V|+|U|)$.

Now, define $G_0=G^*$ and $U_0=U$.
Moreover, let $E_0$ be the subset of $E(G_0)$
in which every edge is incident to at least one vertex from $U_0$.
Then $|E_0|\leq e(G)-e(L,V)\leq k|U_0|$, as $e(L,V)=|L||V|=k|V|$.
It follows that $\sum_{u\in U_0}d_{G_0}(u)\leq2|E_0|\leq2k|U_0|$,
which implies that there exists a vertex $u_0\in U_0$
such that $d_{G_0}(u_0)\leq 2k$.

Then, let $G_1=G_0-\{u_0\}$,
$U_1=U_0\setminus\{u_0\}$
and $E_1$ be the subset of $E(G_1)$
in which every edge is incident to some vertices from $U_1$.
Similarly as above, we have
$e(G_1)\leq k(|V|+|U_1|)$ and
$|E_1|\leq e(G_1)-e(L,V)\leq k|U_1|$.
Thus, we can find a vertex $u_1\in U_1$
such that $d_{G_1}(u_1)\leq 2k$.
Consequently, we can obtain a vertex ordering
$u_0,u_1\ldots,u_{t-1}$ such that
$G_{i}=G_{i-1}-\{u_{i-1}\}$,
$U_{i}=U_{i-1}\setminus\{u_{i-1}\}$
and $d_{G_i}(u_i)\leq 2k$
for each $i\in\{1,\ldots,t-1\}$.
For simplicity, we denote
$d_L(u_i)=d_i$ and $d_{G_i-L}(u_i)=d_i'$.
Then $d_i\leq k-1$ by the definition of $U$,
and $d_i+d_i'=d_{G_i}(u_i)\leq 2k$
for $i\in\{0,\ldots,t-1\}$.

We shall construct a new graph $G$ from $G^*$ as follows.
For each vertex $u_i$ $(0\leq i\leq t-1)$,
we delete all $d_i'$ edges from $u_i$ to $V(G_i-L)$,
and then add all possible $k-d_i$ edges from $u_i$ to $L$.
Denote $\overline{N}_L(u_i)=L\setminus N_{L}(u_i)$.
Then, we can see that
\begin{equation}\label{eq-2.17}
\begin{array}{ll}
\rho(G)-\rho(G^*)
\geq\!\!\!\sum\limits_{uv\in E(G)}\!\!\!x_ux_v-
\!\!\!\sum\limits_{uv\in E(G^*)}\!\!\!x_ux_v
=\sum\limits_{i=0}^{t-1}x_{u_i}\Big(\sum\limits_{v\in\overline{N}_L(u_i)}\!\!\!x_v
-\!\!\!\sum\limits_{v\in N_{G_i-L}(u_i)}\!\!\!x_v\Big).
\end{array}
\end{equation}
Recall that $x_v\geq\big(1-\frac1{2k}\big)x_{u^*}$ for each $v\in L$.
Moreover, since we choose $L=L_{\frac1{2k}}$, it is obvious that
$x_v<\frac1{2k}x_{u^*}$ for each $v\notin L$.
In view of (\ref{eq-2.17}),
we obtain
\begin{equation*}
\begin{array}{ll}
\rho(G)-\rho(G^*)
\geq\sum\limits_{i=0}^{t-1}x_{u_i}x_{u^*}
\Big((k-d_i)(1-\frac1{2k})-d_i'\cdot\frac1{2k}\Big).
\end{array}
\end{equation*}
Recall that $d_i+d_i'\leq 2k$
and $d_i\leq k-1$ for each $i\in\{0,\ldots,t-1\}$.
Thus,
\begin{equation*}
\begin{array}{ll}
(k-d_i)(1-\frac1{2k})-d_i'\cdot\frac{1}{2k}
\geq(k-d_i)(1-\frac1{2k})-(2k-d_i)\frac1{2k}
\geq1-\frac{k+2}{2k}>0.
\end{array}
\end{equation*}
It follows that $\rho(G)>\rho(G^*).$

Observe that $N_G(u_i)=L$ for each $u_i\in U$.
We will further see that $G\cong K_{k,n-k}$.
Indeed, otherwise,
$G\ncong K_{k,n-k}$, then either $e_G(L)\neq0$ or $e_G(V)\neq0$.
However, $G^*[L\cup V]$ contains a spanning subgraph $K_{|L|,|V|}$,
where $|L|=k$ and $|L\cup V|\geq\frac n3.$
Hence, $G^*[L\cup V]$ is clearly $k$-(edge)-connected.
By Lemma \ref{lem-2.1},
$G^*[L\cup V]$ is minimally $k$-(edge)-connected,
which implies that $G^*[L\cup V]\cong K_{|L|,|V|}$.
Since $G[L\cup V]=G^*[L\cup V]$, we
have $G[L\cup V]\cong K_{|L|,|V|}$,
and thus $e_G(L)=e_G(V)=0$, a contradiction.
Hence, $G\cong K_{k,n-k}$.
But now, the inequality $\rho(G^*)<\rho(G)$
contradicts the assumption that $G^*$ has maximal spectral radius.
Therefore, $U=\varnothing$ and $G^*\cong K_{k,n-k}.$
This completes the proof.
\end{proof}

\end{document}